\documentclass[12pt]{amsart}
\usepackage{amssymb}
\usepackage{amsmath}
\usepackage{amsthm}
\usepackage[latin1]{inputenc}
\usepackage[T1]{fontenc}
\usepackage[usenames,dvipsnames,svgnames,table]{xcolor}
\usepackage[colorlinks=true, linkcolor=MidnightBlue, urlcolor= BrickRed, citecolor=MidnightBlue]{hyperref}
\usepackage{rotating}
\usepackage{todonotes}
\usepackage{layout}

\newcommand{\N}{\mathbb{N}}
\newcommand{\Z}{\mathbb{Z}}
\newcommand{\R}{\mathbb{R}}
\newcommand{\C}{\mathbb{C}}
\newcommand{\quat}{\mathbb{H}}
\renewcommand{\S}{{\mathrm{S}}^{4n-1}}
\newcommand{\Hab}{\mathcal{H}^{\alpha,\beta}(\C^{2n})}
\newcommand{\Vll}{V_n^{l,l'}}
\newcommand{\SpnC}{\ensuremath{\mathrm{Sp}(n,\C)}}

\newcommand{\HC}{\ensuremath{\mathrm{H}_{\C}^{2m+1}}}
\renewcommand{\Re}{\mathop{\rm{Re}}}
\newcommand{\scal}[2]{\langle #1,#2\rangle}

\theoremstyle{plain}
\newtheorem{theorem}{Theorem}

\newtheorem{proposition}{Proposition}[section]

\theoremstyle{definition}
\newtheorem{definition}{Definition}[section]

\newtheorem{remark}{Remark}[section]

\begin{document}


\title[Trilinear forms for degenerate principal series of \SpnC]{Invariant trilinear forms for spherical degenerate principal series of complex symplectic groups}
\author{Pierre Clare}
\address{Pierre Clare\\ Department of Mathematics\\ Dartmouth College\\ HB 6188\\Hanover, NH - 03755\\USA}
\email{clare@math.dartmouth.edu}

\begin{abstract}We compute generalized Bernstein-Reznikov integrals associated with standard complex symplectic forms by studying Knapp-Stein intertwining operators between spherical degenerate principal series of $\SpnC$.\end{abstract}

\maketitle


\section*{Introduction}

Let $(\pi_1,V_1)$, $(\pi_2,V_2)$ and $(\pi_3,V_3)$ be continous complex representations of a topological group $G$. An \emph{invariant trilinear form} is a trilinear map $\Psi:V_1\times V_2\times V_3\longrightarrow\C$ such that \[\Psi\left(\pi_1(g)\cdot v_1,\pi_2(g)\cdot v_2,\pi_3(g)\cdot v_3\right)=\Psi(v_1,v_2,v_3)\] for all $g\in G$ and $v_i\in V_i$.

Interest for invariant trilinear forms attached to representations of Lie groups was sparked by the study of fusion rules, that is, the spectral decomposition of tensor product of irreducible representations. See \cite{OksakTrilin, RepkaSL2, MolchTensor, LokeTrilin} for results in the case of $\mathrm{SL}(2,\R)$. More generally, the understanding of these forms, in particular of their linear dependence relations, is involved in the study of multiplicities in branching laws \cite{KobOshFinite}. They are also known to play an important role in the study of composition formulas in quantization theory. See for instance \cite{Unterbergerlivre} and  \cite{KOPUWeylcalc}.

Invariant trilinear forms appeared in connection with automorphic representation theory in the work of J. Bernstein and A. Reznikov, who used them in \cite{BernstRezOriginal} to obtain estimates of automorphic forms. Their result relies on the asymptotic study of certain integral on $\mathrm{S}^1\times\mathrm{S}^1\times\mathrm{S}^1$ for which a closed formula is given in terms of the Euler's function $\Gamma$. It corresponds to the evaluation of the essentially unique invariant trilinear form associated with spherical principal series of $\mathrm{PGL}(2,\R)$ on triples of $\mathrm{PO}(2)$-fixed vectors.

This setting was generalized to spherical principal series representations of more general families of Lie groups: invariant trilinear forms associated with these representations give rise to integrals on triple flag varieties when evaluated on $K$-fixed vectors. Such integrals are called \emph{generalized Bernstein-Reznikov integrals}.

The case of Lorentz groups $\mathrm{SO}_\circ(d,1)$ was studied separately in \cite{DeitmarTriple} and \cite{BernsteinRez}. Closed formulas for other generalized Bernstein-Reznikov integrals were obtained in \cite{BernsteinRez} in relation with the representation theory of different families of groups, namely $\mathrm{Sp}(n,\R)$ and $\mathrm{GL}(n,\R)$. The method of \cite{BernsteinRez} also allowed to recover the case of Lorentz groups from another point of view. A complementary study of these conformally invariant trilinear forms was done in \cite{BeckClerc} where residues at singular parameters were computed.

Finally, a systematic study of invariant trilinear forms constructed from Knapp-Stein intertwiners attached to spherical principal series representations of semisimple Lie groups of real-rank 1 groups was achieved in \cite{BKZTrilinear}, where a description of the residues at some poles in terms of bidifferential operators is obtained. Note also that generalizations to symmetric pairs are considered in \cite{KobSpehCRAS} and \cite{KobSpehMemAMS}, with different goals.

The general framework in which these integrals are related to Representation Theory is described in Section 5 of \cite{BernsteinRez}, where the important role played by Knapp-Stein intertwiners and their normalisations is highlighted: generalized Bernstein-Reznikov integrals are interpreted as traces, which can in turn be evaluated when these operators are understood in sufficient detail. More precisely, the heart of the computation of these traces is the explicit determination of the $K$-spectrum of the intertwiners.

The work presented here follows this method. It relies on results about degenerate principal series of $\SpnC$ obtained in \cite{PArtSpnC}, where normalized Knapp-Stein intertwiners were described in terms of complex symplectic Fourier transformations, leading to the $K$-type analysis needed to evaluate the corresponding generalized Bernstein-Reznikov integrals.

\subsection*{Main result}

The purpose of this note is to study, for $n\geq1$ integer, the following integral: \[\mathcal{I}_n=\hspace{-7mm}\int\limits_{\S\times\S\times\S}\hspace{-9mm}\left|\Re\omega(y,z)\right|^{\frac{\alpha}{2}-n}\left|\Re\omega(z,x)\right|^{\frac{\beta}{2}-n}\left|\Re\omega(x,y)\right|^{\frac{\gamma}{2}-n}\,d\tilde{\sigma}(x,y,z)\] 
where $\omega$ denotes the standard complex sympletic form on $\C^{2n}$ and $\tilde{\sigma}$ is the cube of the euclidean measure $\sigma_{\S}$ on the $(4n-1)$-dimensional sphere: \[\tilde{\sigma}=\sigma_{\S}\otimes\sigma_{\S}\otimes\sigma_{\S}.\]

More precisely, we establish the following result.
\begin{theorem}\label{main}As a meromorphic function of the parameters $\alpha$, $\beta$ and $\gamma$, the integral $\mathcal{I}_n$ is equal to
\[A(\alpha,\beta,\gamma)\,.\,{}_6F_5\left(\begin{array}{ccccccl}2n-1\;,&2n-1\;,&n+\frac{1}{2}\;,&\frac{2n-\alpha}{4}\;,&\frac{2n-\beta}{4}\;,&\frac{2n-\gamma}{4}\;,&\\&&&&&&;1\\&1\;,&n-\frac{1}{2}\;,&\frac{6n+\alpha}{4}\;,&\frac{6n+\beta}{4}\;,&\frac{6n+\gamma}{4}\;,&\end{array}\right)\]

with \[A(\alpha,\beta,\gamma)=8\pi^{6n-\frac{3}{2}}\frac{\Gamma\left(\frac{2-2n+\alpha}{4}\right)\Gamma\left(\frac{2-2n+\beta}{4}\right)\Gamma\left(\frac{2-2n+\gamma}{4}\right)}{\Gamma\left(\frac{6n+\alpha}{4}\right)\Gamma\left(\frac{6n+\beta}{4}\right)\Gamma\left(\frac{6n+\gamma}{4}\right)}.\]
\end{theorem}
The definition of the hypergeometric function ${}_pF_q$ is recalled p.\pageref{hypergeomdef} and the precise domain of convergence for $\mathcal{I}_n$ is given by Proposition \ref{domconv}.

\subsection*{Outline}

Notations and definitions are fixed in Section \ref{sectionP}, where we also recall useful facts about real, complex and quaternionic spherical harmonics. Sections \ref{sectionDPSKT} and \ref{sectionIO} are devoted to the discussion of spherical degenerate principal series of $\SpnC$ and the corresponding Knapp-Stein intertwiners. In particular, the $K$-spectrum of these operators is given in Proposition \ref{All'}. Theorem \ref{main} is proved in Section \ref{sectionTI} and the domain of convergence is determined in Proposition \ref{domconv}. Finally, we use classical and more recent reduction results on well-poised hypergeometric functions to derive closed formulas when $n=1$ and $n=2$ in Section \ref{SecCF}.

\section{Preliminaries}\label{sectionP}

The representation theory of complex symplectic groups is naturally formulated in terms of quaternionic and complex vector spaces. However, it will often be convenient to deal with functions of real variables and their integral transforms. The next paragraphs are devoted to fixing identifications and introducing the relevant functional spaces and transformations to be used in the proof of the main result.

\subsection{Identifications and bilinear forms}

Let $\R$, $\C$ and $\quat$ be the fields of real, complex and quaternionic numbers respectively, $i$ and $j$ the usual anticommuting complex and quaternionic units. For any integer $n\geq1$, we fix the following identifications: \begin{equation}\label{RCH}\quat^n\simeq\left(\C^n+j\C^n\right)\simeq\left(\left(\R^n+i\R^n\right)+j\left(\R^n+i\R^n\right)\right).\end{equation}
The identity matrix on $\mathbb{K}^p$, with $\mathbb{K}$ any of the above fields, is denoted by $I_p$. The complex conjugation on $\C^{2n}=\R^{2n}+i\R^{2n}\simeq\R^{2n}\times\R^{2n}$ is implemented by the matrix: \[\varepsilon = \left[\begin{array}{c|c}I_{2n}&0\\\hline \;0\;&-I_{2n}\end{array}\right],\] and if $f$ is a function on $\R^{4n}$, we denote by $f^\varepsilon$ the function $f(\varepsilon \,\cdot)$.

For $X=(x_1,\ldots,x_{2n})$ and $Y=(y_1,\ldots,y_{2n})$ in $\C^{2n}$, denote \[\scal{X}{Y}=\sum_{k=1}^{2n}x_k y_k.\] The hermitian inner product on $\C^{2n}$ is thus defined by \begin{equation}\label{scal}(X,Y)=\scal{X}{\bar{Y}}=\scal{X_\R}{\varepsilon Y_\R}\end{equation} where $X_\R$ and $Y_\R$ are the images of $X$ and $Y$ under the identification $\C^{2n}\simeq\R^{4n}$ given by \eqref{RCH}.

The symplectic structure on $\C^{2n}$ is given by the matrix \[J=\left[\begin{array}{c|c}0&-I_n\\\hline \;I_n\;&0\end{array}\right]\] corresponding to the element $j\in\quat$.

The \emph{symplectic form} $\omega$ is defined by $\omega(X,Y)=\scal{X}{JY}$, that is, \[\omega(X,Y)=\scal{X_2}{Y_1}-\scal{X_1}{Y_2}.\] The \emph{complex symplectic group} is the group of complex invertible matrices of size $2n$ preserving $\omega$: \[\SpnC=\left\{g\in\mathrm{GL}(2n,\C)\;\left|\;\forall X,Y\in\C^{2n},\right.\;\omega(gX,gY)=\omega(X,Y)\right\}.\] Equivalently, the elements $g$ of $\mathrm{GL}(2n,\C)$ belonging to $\SpnC$ are exactly those that satisfy the relation $^{t}gJg=J$.

\subsection{Integral transforms}\label{IT}

Most of the computations in this paper rely on the use of integral transformations, similar to the classical Fourier transform on euclidean spaces. Here, we collect the various operators that will be used and the relations between them. The formulas initially make sense for Schwartz functions and extend to operators on the space of square-integrable functions.

Consider the \textbf{complex Fourier transform on $\C^{2n}$:}
\begin{equation*}\mathcal{F}_{\C^{2n}}f(\xi)=\int_{\C^{2n}}f(X)e^{-2i\pi\Re{\scal{X}{\xi}}}\,dX.\end{equation*}
It allows to define the \textbf{complex symplectic Fourier transform} by
\begin{equation}\label{Fsymp}\mathcal{F}_{\mathrm{symp}}f(\xi)=\mathcal{F}_{\C^{2n}}f(J\xi),\end{equation}
that is,
\begin{equation*}\mathcal{F}_{\mathrm{symp}}f(\xi)=\int_{\C^{2n}}f(X)e^{-2i\pi\Re{\omega(X,\xi)}}\,dX.\end{equation*}

The \textbf{real conjugate Fourier transform on $\R^{4n}$} is defined by
\begin{equation}\label{Feps}\mathcal{F}_\varepsilon\,f(\xi) = \mathcal{F}_{\R^{4n}}f(\varepsilon\xi),\end{equation}
so that
\begin{equation*}\mathcal{F}_\varepsilon\,f(\xi_1,\xi_2) = \mathcal{F}_{\C^{2n}}\,f(\xi_1+i\xi_2).\end{equation*}
In other terms,
\begin{equation*}\mathcal{F}_\varepsilon\,f(\xi) = \int_{\R_1^{2n}\times\R_2^{2n}}f(X)e^{-2i\pi\left(\scal{X_1}{\xi_1}-\scal{X_2}{\xi_2}\right)}\,d(X_1,X_2).\end{equation*}
As explained above, under the identification $\R_1^{2n}\times\R_2^{2n}\simeq\C^{2n}$ given by $(X_1,X_2)\longmapsto X_1 + iX_2$, the matrix $\varepsilon$ induces the complex conjugation. By \eqref{scal}, the complex Fourier transform can be seen as the transform $\mathcal{F}_\varepsilon\,f = \left(\mathcal{F}_{\R^{4n}}\,f\right)^\varepsilon$. Indeed, if $\xi = (\xi_1,\xi_2)$ and $\zeta=\xi_1 + i\xi_2$, it is clear that \begin{equation*}\mathcal{F}_{\C^{2n}}\,f(\zeta) = \mathcal{F}_\varepsilon\,f(\xi).\end{equation*}

Therefore, the complex symplectic Fourier transform \eqref{Fsymp} is related to the real conjugate Fourier transform \eqref{Feps} by the following formula: \begin{equation}\label{sympj}\mathcal{F}_{\mathrm{symp}}f=\mathcal{F}_{\R^{4n}}f(\cdot\,j).\end{equation}

According to the above definitions, the left-hand side of this identity is a function of a complex variable, whereas the one on the right is a function of a real variable. This ambiguity is resolved by the identifications \eqref{RCH}, which also allow the multiplication by the quaternionic unit $j$. The fact that this multiplication occurs on the right although the definition of $\mathcal{F}_{\mathrm{symp}}$ involves the matrix $J$ acting from the left is due to the complex conjugation and the relation $j\bar{z}=zj$, true for any complex number $z$.

\subsection{Spherical harmonics}\label{SH}

The $K$-type analysis of the representations involved in our computations mainly relies on the decomposition of square-integrable functions on the sphere $\S$. Under the identifications of \eqref{RCH}, $\S$ can be seen as the unit sphere $S_1$ in $\quat^n$, $\C^{2n}$ or $R^{4n}$. Moreover, these identifications are compatible with the standard actions of the orthogonal, unitary and symplectic groups, as well as that of scalars of modulus one, namely $\{\pm1\}\subset\R$, $\mathrm{U}(1)\subset\C$ and $\mathrm{Sp}(1)\subset\quat$. Note that we let scalars act on the right, in order to avoid confusion in the non-commutative quaternionic case.

The situation is summarized in the following diagram and the relevant quotients shall be described in Section \ref{sectionDPSKT}.

\begin{equation*}\begin{array}{ccccl}
\mathrm{Sp}(n)&\curvearrowright&S_1(\quat^n)&\curvearrowleft&\mathrm{Sp}(1)\simeq\mathrm{SU}(2)\\
\cap&&\begin{turn}{90}$\simeq$\end{turn}&&\:\cup\\
\mathrm{U}(2n)&\curvearrowright&S_1(\C^{2n})&\curvearrowleft&\mathrm{U}(1)\\
\cap&&\begin{sideways}$\simeq$\end{sideways}&&\:\cup\\
\mathrm{O}(4n)&\curvearrowright&S_1(\R^{4n})&\curvearrowleft&\left\{\pm1\right\}\\
&&\begin{sideways}$\simeq$\end{sideways}&&\\
&&\S&&\\
\end{array}\end{equation*}

The compatibility of the identifications allows us to consider functions on the sphere as depending on real, complex or quaternionic variables as needed. With this general picture in mind, let us recall some facts that will be of use further on. Proofs and details can be found in \cite{BernsteinRez} and \cite{PArtSpnC}. Following classical notations, we denote by $\mathcal{H}^k(\R^{4n})$ the space of harmonic homogeneous polynomials of degree $k$ on $\R^{4n}$, which is naturally a representation of $\mathrm{O}(4n)$. Analogously, for $\alpha$ and $\beta$ integers, $\Hab$ will denote the space of polynomials $P(Z,\bar{Z})$ on $\C^{2n}$ that are homogeneous of degree $\alpha$ in $Z=(z_1,\ldots,z_{2n})$, of degree $\beta$ in $\bar{Z}=(\bar{z}_1,\ldots,\bar{z}_{2n})$ and subject to the condition \[\sum_{j=1}^{2n}\frac{\partial^2p}{\partial z_j\partial \bar{z}_j}=0.\]

By definition, there is a natural isomorphism of vector spaces that commutes to the $\mathrm{U}(2n)$-actions:
\begin{equation}\label{HkHab}\mathcal{H}^k(\R^{4n})\left|_{\mathrm{U}(2n)}\right.\simeq\bigoplus_{\alpha+\beta=k}\Hab\end{equation}


In the quaternionic case, we denote by $\Vll$ the unique unitary irreducible representation of $\mathrm{Sp}(n)$ corresponding to the highest weight $(l,l',0,\ldots,0)$ where $l$ and $l'$ are integers satisfying $l\geq l'\geq0$. Similarly, $V_1^j$ denotes the irreducible $j+1$-dimensional representation of $\mathrm{Sp}(1)\simeq\mathrm{SU}(2)$. Analogously to \eqref{HkHab}, one has the following compatibility relation between representations of $\mathrm{Sp}(n)$ and $\mathrm{U}(2n)$:
\begin{equation}\label{HmmVll}\mathcal{H}^{m,m}(\C^{2n})\left|_{\mathrm{Sp}(n)}\right.\simeq\bigoplus_{l+l'=2m}\Vll.\end{equation}

Since $\mathrm{U}(1)$ naturally embeds into $\mathrm{SU}(2)$ \textit{via}\begin{equation*}e^{i\theta}\mapsto\left[\begin{array}{cc}e^{i\theta}&0\\0&e^{-i\theta}\end{array}\right],\end{equation*} $V_1^j$ decomposes according to the characters of $\mathrm{U}(1)$. More precisely, if $\C_\delta$ denotes the space of the character $z\mapsto z^\delta$ for $\delta\in\Z$, then
\begin{equation}\label{VllCdelta}V_1^j\simeq\bigoplus_{\substack{|\delta|\leq j\\\delta\equiv j[2]}}\C_\delta.\end{equation}
The isotypic decomposition of $L^2\left(\S\right)$ with respect to the action of the direct product $\mathrm{Sp}(n)\times\mathrm{Sp}(1)$ is given in~\cite{HoweTan}: \[L^2\left(\S\right)\simeq\sideset{}{^\oplus}\sum_{l\geq l'\geq0} \Vll\otimes V_1^{l-l'}.\] Together with \eqref{VllCdelta}, this result leads to the following decomposition under the action of $\mathrm{Sp}(n)\times\mathrm{U}(1)$: \begin{equation}\label{L2delta}L^2\left(\S\right)\simeq\sideset{}{^\oplus}\sum_{\left(\delta,(l,l')\right)\in\Z\times\N^2\,,\,\left\{\substack{l-l'\geq |\delta|\\l-l'\equiv\delta[2]}\right.}\Vll\otimes\C_\delta.\end{equation}

\section{\texorpdfstring{Spherical degenerate principal series and their $K$-types}{Spherical degenerate principal series and their K-types}}\label{sectionDPSKT}

The complex symplectic group $G=\SpnC$ acts on the projective space $\C\mathrm{P}^{2n-1}$. The stabiliser of a point under this action is a maximal parabolic subgroup \[P=L\bar{N}\simeq\left(\C^\times.\mathrm{Sp}(m,\C)\right)\ltimes\HC\] where $m=n-1$. For any complex number $\lambda$, the character $a\mapsto|a|^\lambda$ of $\C^\times$ can be extended to a character $\chi_\lambda$ of $P$, then induced to $G$:\[\pi_{\lambda}=\mathop{\rm{Ind}}\nolimits_{P}^{G}{\chi_\lambda}.\]
The family $\left\{\pi_\lambda\right\}_{\lambda\in i\R}$ is called the \textit{spherical degenerate principal series} of $\SpnC$, and was investigated in \cite{Gross} and \cite{PArtSpnC}. In particular, $\pi_\lambda$ is irreducible for $\lambda\neq0$, representations with opposite parameters $\pi_\lambda$ and $\pi_{-\lambda}$ are equivalent and $\pi_0$ splits into the direct sum of two irreducible components. We recall here more precise facts about these representations and the intertwining relations between them that will be of later use.

The representation $\pi_\lambda$ can be realised on a space that admits \[V_{\lambda}^\infty = \left\{f\in C^\infty(\C^{2n}\setminus\left\{0\right\})\,\left|\,\forall a\in\C^\times\,,\,f(a\,\cdot)=|a|^{-\lambda-2n}f\right.\right\}\] as a dense subspace. By homogeneity, functions in $V_{\lambda}^\infty$ are determined by their restriction to the euclidean sphere $\S$. The compact picture of $\pi_{\lambda}$ is obtained by restriction to $\S$ according to the identifications \begin{equation*}K/(L\cap K)\simeq\mathrm{Sp}(n)/\mathrm{U}(1).\mathrm{Sp}(m)\simeq\S/\mathrm{U}(1)\end{equation*} and $\pi_{\lambda}$ is realised on $L^2\left(\S\right)$. Using quaternionic spherical harmonics as described in Preliminaries \ref{SH}, more precisely fixing $\delta=0$ in \eqref{L2delta}, this leads to the $K$-type decomposition \begin{equation}\label{Ktypes}\left.\pi_{\lambda}\right|_{\mathrm{Sp}(n)}\simeq\sideset{}{^\oplus}\sum_{\substack{l\geq l'\geq0\\l\equiv l'[2]}} \Vll.\end{equation}

\begin{remark}This formula is a special case of the $K$-type decomposition of non-spherical degenerate principal series of $\SpnC$ given in \cite[Proposition 2]{PArtSpnC}.\end{remark}

\section{Intertwining operators}\label{sectionIO}

Intertwining operators among the degenerate principal series of $\SpnC$ were studied in \cite{PArtSpnC}, where normalization was obtained by means of the complex symplectic Fourier transform and the $K$-spectrum was explicitely computed. Let us recall here the facts that will enter the determination of invariant trilinear forms.

\subsection{Normalization of the Knapp-Stein intertwiners}

If $f$ is a square-integrable function $f$ on $\S$ and $\Re(\lambda)<-2n+1$, the integral \begin{equation*}T_{\lambda} f(y)=\int_{\S} f(x)\left|\Re(\omega(x,y))\right|^{-\lambda-2n}\,d\sigma(x)\end{equation*} is absolutely convergent, thus yielding a meromorphic family $\left\{T_\lambda\right\}$ of operators that are easily seen to intertwine the actions of $\SpnC$ on $L^2\left(\S\right)$ seen as the compact picture for $\pi_{-\lambda}$ and $\pi_{\lambda}$ respectively. Therefore, $T_\lambda$ is called the non-normalized \emph{Knapp-Stein intertwiner} between $\pi_{-\lambda}$ and $\pi_{\lambda}$.

The corresponding normalized intertwiner is given by the symplectic Fourier transform of Preliminaries \ref{IT}. More precisely, it is proved in \cite[Proposition 1]{PArtSpnC} that:

\begin{equation}\label{Tsymp}T_{\lambda}=2\pi^{\lambda+2n-\frac{1}{2}}\frac{\Gamma\left(\frac{1-\lambda}{2}-n\right)}{\Gamma\left(\frac{\lambda}{2}+n\right)}\mathcal{F}_{\mathrm{symp}}\left|_{V_{-\lambda}^\infty}.\right.\end{equation}

\subsection{\texorpdfstring{$K$-spectrum of $T_\lambda$}{K-spectrum of Tlambda}}

Each summand occurs with multiplicity 1 in the $K$-type formula \eqref{Ktypes} so the intertwiner $T_\lambda$ acts on each $\Vll$ by a scalar. The following result specifies the eigenvalues in this `diagonalisation'.

\begin{proposition}\label{All'} The Knapp-Stein intertwiner $T_\lambda$ acts on $\Vll$ by \[A_{l+l'}(\lambda)=2\pi^{2n-\frac{1}{2}}\frac{\Gamma\left(\frac{1-\lambda}{2}-n\right)}{\Gamma\left(\frac{\lambda}{2}+n\right)}\frac{\Gamma\left(\frac{l+l'+\lambda}{2}+n\right)}{\Gamma\left(\frac{l+l'-\lambda}{2}+n\right)}.\]
\end{proposition}

\begin{proof} In view of the compatibility between the isotypic decompositions of Preliminaries \ref{SH}, an element of $\Vll$ can be seen as the restriction to the unit sphere of a polynomial in $\mathcal{H}^{l+l'}(\R^{4n})$.

Consider $p$ in $\mathcal{H}^k(\R^{4n})$. Its restriction $\left.p\right|_{\S}$ can be extended to a homogeneous function $p_\alpha$ on $\R^{4n}$ by \[p_\alpha(r\omega)=r^\alpha p(\omega).\] The image of $p_\alpha$ under the Fourier transform is given by the following identity between distributions on $\R^{4n}$ depending meromorphically on $\alpha$, proved in \cite[Lemma 2.3]{BernsteinRez}: \begin{equation}\label{eqFalpha}\mathcal{F}_{\R^{4n}}\,p_\alpha = \pi^{-\alpha-2n}i^{-k}\frac{\Gamma\left(2n + \frac{k+\alpha}{2}\right)}{\Gamma\left(\frac{k-\alpha}{2}\right)}\,p_{-\alpha-4n}.\end{equation}

Using the expression of $T_\lambda$ in terms of the complex symplectic Fourier transform \eqref{Tsymp} and the relation between the latter and the real Fourier transform \eqref{sympj} on $\R^{4n}$, letting $\alpha=\lambda-2n$ and $k=l+l'$ in \eqref{eqFalpha} leads to: \begin{equation*}\operatorname{T_{\lambda}}p_{\lambda-2n}=2\pi^{2n-\frac{1}{2}}i^{-(l+l')}\frac{\Gamma\left(\frac{1-\lambda}{2}-n\right)}{\Gamma\left(\frac{\lambda}{2}+n\right)}\frac{\Gamma\left(\frac{l+l'+\lambda}{2}+n\right)}{\Gamma\left(\frac{l+l'-\lambda}{2}+n\right)}p_{-(\lambda-2n)}(\cdot j)\end{equation*}

The action of the right multiplication by $j$ in the last formula can be determined by seeing $\Vll$ as a subspace of $\mathcal{H}^{\frac{l+l'}{2},\frac{l+l'}{2}}\left(\C^{2n}\right)$, according to Preliminaries \ref{SH}. More precisely, the discussion of Section 6.3 of \cite{HoweTan} implies that the polynomial $\left(z_1\bar{w_1}\right)^{\frac{l-l'}{2}}\left(z_2\bar{w_1}-z_1\bar{w_2}\right)^{l'}$ is a highest weight vector in $\Vll$. It implies that $p_{\lambda-2n}(\cdot j)=(-1)^{\frac{l+l'}{2}}p_{\lambda-2n}$ for any $p$ in $\Vll$ and the result follows.\end{proof}

\section{Triple integrals}\label{sectionTI}

Let us now turn to the proof of Theorem \ref{main}. Following \cite{BernsteinRez}, our strategy will be to interpret $\mathcal{I}_n$ as the trace of an operator build from the intertwiners discussed above. For $j\in\left\{1,2,3\right\}$, we let \begin{equation*}\mu_j=\frac{\lambda_1+\lambda_2+\lambda_3-2n}{2}-\lambda_j\end{equation*} and denote by $\underline{\mu}$ the triple $(\mu_3,\mu_2,\mu_3)$.

Each $T_{\mu_j}$ is a Hilbert-Schmidt operator on $L^2\left(\S\right)$ for $\Re\mu_j<<0$, so the composition $T_{\mu_1}T_{\mu_2}T_{\mu_3}$ is of trace class and its trace $\tau_n$ is given by \begin{eqnarray*}\tau_n&=&\mathrm{Trace}\left(T_{\mu_1}T_{\mu_2}T_{\mu_3}\right)\\&=&\int\limits_{\left(\S\right)^3}\left|\Re\omega(y,z)\right|^{-\mu_1-2n}\left|\Re\omega(z,x)\right|^{-\mu_1-2n}\left|\Re\omega(x,y)\right|^{-\mu_1-2n}\,d\sigma^3(x,y,z)\end{eqnarray*}where $d\sigma^3(x,y,z)$ denotes $d\sigma(x)\,d\sigma(y)\,d\sigma(z)$.

The trace $\tau_n$ can be computed on the convergence domain, that will be made precise in Paragraph \ref{sectionDOC}, by means of the `diagonalisation' of the previous section.

\subsection{Proof of the main result}

According to the $K$-type formula \eqref{Ktypes} and Proposition \ref{All'}, one has \begin{equation*}\tau_n=\sum_{l\equiv l'[2]}\left(\prod_{j=1}^3 A_{l+l'}(\mu_j)\right)\dim\Vll\end{equation*}
and reorganising the sum with respect to the integer $m=\frac{l+l'}{2}$ gives
\begin{equation}\label{tau}\tau_n=\sum_{m\geq0}\left(\prod_{j=1}^3 A_{2m}(\mu_j)\sum_{l+l'=2m}\dim\Vll\right).\end{equation}
Observe that \[A_{2m}(\mu)=K_{n,\mu}\frac{\Gamma\left(n+m+\frac{\mu}{2}\right)}{\Gamma\left(n+m-\frac{\mu}{2}\right)}\] with $K_{n,\mu}$ independent of $m$, so that \[\frac{A_{2m+2}(\mu)}{A_{2m}(\mu)}=\frac{\Gamma\left(n+m+1+\frac{\mu}{2}\right)\Gamma\left(n+m-\frac{\mu}{2}\right)}{\Gamma\left(n+m+1-\frac{\mu}{2}\right)\Gamma\left(n+m+\frac{\mu}{2}\right)}=\frac{n+m+\frac{\mu}{2}}{n+m-\frac{\mu}{2}}.\]
Using the Pochhammer symbol \[(a)_r=\dfrac{\Gamma(a+r)}{\Gamma(a)}=a(a+1)\ldots(a+r-1),\] it follows that \begin{equation}\label{A2m}A_{2m}=\frac{\left(n+\frac{\mu}{2}\right)_m}{\left(n-\frac{\mu}{2}\right)_m}A_0(\mu)\end{equation} with \begin{equation}\label{A0}A_0(\mu)=2\pi^{2n-\frac{1}{2}}\frac{\Gamma\left(\frac{1-\mu}{2}-n\right)}{\Gamma\left(n-\frac{\mu}{2}\right)}.\end{equation}
The isomorphism \eqref{HmmVll} implies that \[\sum_{l+l'=2m}\dim\Vll=\dim\mathcal{H}^{m,m}\left(\C^{2n}\right).\]
The right-hand side is given by Formula (3.3) in \cite{BernsteinRez}, and it follows that \begin{eqnarray}\sum_{l+l'=2m}\dim\Vll&=&\frac{(2m+2n-1)\left((m+1)_{2n-2}\right)^2}{\Gamma(2n)\Gamma(2n-1)}\nonumber\\\label{carre}&=&\frac{2m+2n-1}{2n-1}\left(\frac{(m+1)_{2n-2}}{\Gamma(2n-1)}\right)^2.\end{eqnarray}
The duplication formula for $\Gamma$ implies \[\label{duplipoch}(a)_{2m}=2^{2m}\left(\frac{a}{2}\right)_m\left(\frac{a+1}{2}\right)_m,\] so observing that \[\frac{2m+2n-1}{2n-1}=\frac{(2n)_{2m}}{(2n-1)_{2m}},\] we get \begin{equation}\label{fact1}\frac{2m+2n-1}{2n-1}=\frac{\left(n+\frac{1}{2}\right)_m}{\left(n-\frac{1}{2}\right)_m}.\end{equation}
To deal with the second factor in \eqref{carre}, it is enough to observe that \begin{equation}\label{fact2}\frac{(m+1)_{2n-2}}{\Gamma(2n-1)}=\frac{\Gamma(m+2n-1)}{\Gamma(m+1)\Gamma(2n-1)}=\frac{(2n-1)_m}{\Gamma(m+1)}.\end{equation}
The definitions imply that $\Gamma(m+1)=m!=(1)_m$, so Formulas \eqref{fact1} and \eqref{fact2} lead to \begin{equation}\label{dimsum}\sum_{l+l'=2m}\dim\Vll=\frac{(2n-1)_m(2n-1)_m\left(n+\frac{1}{2}\right)_m}{m!(1)_m\left(n-\frac{1}{2}\right)_m}.\end{equation}

The discussion above allows to write $\tau_n$ as the sum of a hypergeometric series. 

\begin{definition}\label{hypergeomdef}
If $p$ and $q$ are positive integers, the \emph{hypergeometric function} with parameters $a_1,\ldots,a_p$ and $b_1,\ldots,b_q$ is defined by \[{}_pF_q\left(\begin{array}{ccccl}a_1\;,&a_2\;,&\ldots\;,&a_p&\\&&&&;z\\b_1\;,&b_2\;,&\ldots\;,&b_q&\end{array}\right)=\sum_{k=0}^{+\infty}\frac{(a_1)_k(a_2)_k\ldots(a_p)_k}{(b_1)_k(b_2)_k\ldots(b_q)_k}\frac{z^k}{k!}.\]
\end{definition}

Then, \eqref{tau}, \eqref{A2m}, \eqref{A0} and \eqref{dimsum} imply that
 \begin{equation}\label{tau_n}\tau_n=A_{\underline{\mu}}\cdot{}_6F_5\left(\begin{array}{ccccccl}2n-1\;,&2n-1\;,&n+\frac{1}{2}\;,&n+\frac{\mu_1}{2}\;,&n+\frac{\mu_2}{2}\;,&n+\frac{\mu_3}{2}\;,&\\&&&&&&;1\\&1\;,&n-\frac{1}{2}\;,&n-\frac{\mu_1}{2}\;,&n-\frac{\mu_2}{2}\;,&n-\frac{\mu_3}{2}\;,&\end{array}\right),\end{equation}
with \begin{eqnarray*}A_{\underline{\mu}}&=&A_0(\mu_1)A_0(\mu_2)A_0(\mu_3)\\&=&8\pi^{6n-\frac{3}{2}}\frac{\Gamma\left(\frac{1-\mu_1}{2}-n\right)}{\Gamma\left(n-\frac{\mu_1}{2}\right)}\frac{\Gamma\left(\frac{1-\mu_2}{2}-n\right)}{\Gamma\left(n-\frac{\mu_2}{2}\right)}\frac{\Gamma\left(\frac{1-\mu_3}{2}-n\right)}{\Gamma\left(n-\frac{\mu_3}{2}\right)},\end{eqnarray*}
and Theorem \ref{main} follows from letting $\alpha=\lambda_1-\lambda_2-\lambda_3=-2(n+\mu_1)$, $\beta=-\lambda_1+\lambda_2-\lambda_3=-2(n+\mu_2)$ and $\gamma=-\lambda_1-\lambda_2+\lambda_3=-2(n+\mu_3)$.

\subsection{Domain of convergence}\label{sectionDOC}

The integrals discussed above are originally defined for parameters $(\mu_1,\mu_2,\mu_3)$ in regions that guarantee absolute convergence, then extended meromorphically with respect to these parameters. In order to give an exact description of the convergence region, the method used in \cite{BernsteinRez} can be directly applied. In order to use real coordinates, we identify $\C^{2n}$ and $\R^{4n}$ \textit{via} \[(X_1^1+iX_1^2,X_2^1+iX_2^2)\longmapsto(X_1^1,X_2^1,X_1^2,X_2^2)\] and let \[\Phi=\left[\begin{array}{c|c}J&0\\\hline \;0\;&-J\end{array}\right].\] One then gets, for $X,Y\in\R^{4n}$: \[\Re\omega(X,Y)=\scal{X}{\Phi Y},\] and, considering \[h_{\underline{\lambda}}(X,Y,Z)=\left|\scal{Y}{\Phi Z}\right|^{\lambda_1}\left|\scal{Z}{\Phi X}\right|^{\lambda_2}\left|\scal{X}{\Phi Y}\right|^{\lambda_3}\] with $\underline{\lambda}=(\lambda_1,\lambda_2,\lambda_3)$ and replacing $J$ by $\Phi$ everywhere in the proof of Proposition 6.2 of \cite{BernsteinRez} directly leads to the following:

\begin{proposition}\label{domconv}
The integral \[\int\limits_{\S\times\S\times\S}h_{\underline{\lambda}}(X,Y,Z)\,d\sigma(X)\,d\sigma(Y)\,d\sigma(Z)\] converges absolutely if and only if $\Re(\lambda_j)>-1$ for $j\in\left\{1,2,3\right\}$.
\end{proposition}

\section{Closed formulas}\label{SecCF}

The equality \eqref{tau_n} allows to express the generalized Bernstein-Reznikov integral $\mathcal{I}_n$ in terms of a hypergeometric function, namely of the form ${}_6F_5$. Under certain conditions on the parameters, such functions are said \emph{well-poised} and special values can sometimes be evaluated as quotients of expressions involving Euler's function $\Gamma$. 

Such formulas have been obtained in the conformal case in \cite{DeitmarTriple}, \cite{ClercOrsted}, \cite{BernsteinRez} and \cite{BKZTrilinear} and further results were recently established by J.-L. Clerc in \cite{1411.3610,1507.1470}. Similarly, closed formulas were found in the case of real symplectic groups in \cite{BernsteinRez}, where the methods used here find their origin. See also the composition formulas obtained in pseudo-differential calculus related to the quantization of certain homogeneous spaces \cite{PevznerUnter,Unterbergerlivre}.

On the contrary, the Euclidean case, corresponding to the real general linear group studied in \cite{BernsteinRez} and the complex, quaternionic and octonionic real-rank 1 cases discussed in \cite{BKZTrilinear} give rise to sums of generalized hypergeometric series for which no closed formulas are currently available.

Beyond the prospect of discovering new reduction formulas for hypergeometric functions, the interpretation of the existence of closed formulas for generalized Bernstein-Reznikov integrals (or lack thereof) in geometric terms is an interesting problem. Aspects of this problem, such as the generic uniqueness of invariant trilinear forms are discussed in Section 5.1 of \cite{BernsteinRez}. It is related to the study of orbits of the diagonal action of the group $G$ on certain triple flag varieties $G/P\times G/P\times G/P$ which generalize the product of circles that appears in the original work of Bernstein and Reznikov. See the introduction of \cite{ClercNeeb} and \cite{1411.3610,1507.1470} for a review of known results and extensions of these ideas to a broader context, such as that of the Shilov boundary of certain symmetric domains \cite{TripleShilov}.

In the remainder of this paragraph, we derive and discuss closed formulas in the complex symplectic case for $n=1$ and $n=2$.

\subsection{\texorpdfstring{Case $n=1$}{Case n=1}} Let us go back to the general formula \eqref{tau} giving $\tau_n$ as a trace. Fixing $l'=0$ implies $m=\frac{l}{2}$ and \[\dim\Vll=\dim V_1^{2m}=2m+1,\] so \eqref{tau} becomes \begin{equation*}\tau_1=A_{\underline{\mu}}\cdot\sum_{m\geq0}\frac{\left(1+\frac{\mu_1}{2}\right)_m}{\left(1-\frac{\mu_1}{2}\right)_m}\frac{\left(1+\frac{\mu_2}{2}\right)_m}{\left(1-\frac{\mu_2}{2}\right)_m}\frac{\left(1+\frac{\mu_3}{2}\right)_m}{\left(1-\frac{\mu_3}{2}\right)_m}(2m+1)\end{equation*}
with \[A_{\underline{\mu}}=A_0(\mu_1)A_0(\mu_2)A_0(\mu_3)=8\pi^{\frac{9}{2}}\frac{\Gamma\left(-\frac{1+\mu_1}{2}\right)}{\Gamma\left(1-\frac{\mu_1}{2}\right)}\frac{\Gamma\left(-\frac{1+\mu_2}{2}\right)}{\Gamma\left(1-\frac{\mu_2}{2}\right)}\frac{\Gamma\left(-\frac{1+\mu_3}{2}\right)}{\Gamma\left(1-\frac{\mu_3}{2}\right)}.\]

The final form of the result is \begin{equation*}\tau_1=A_{\underline{\mu}}\,.\,{}_5F_4\left(\begin{array}{rrrrrl}1,&\frac{3}{2},&1+\frac{\mu_1}{2},&1+\frac{\mu_2}{2},&1+\frac{\mu_3}{2}&\\&&&&&;1\\&\frac{1}{2},&1-\frac{\mu_1}{2},&1-\frac{\mu_2}{2},&1-\frac{\mu_3}{2}&\end{array}\right).\end{equation*}

A closed formula can be obtained  for $\tau_1$ by means of the Dougall-Ramanujan formula given in \cite[p.416]{BernsteinRez} (see also \cite[pp.25-26]{Bailey}), namely \begin{equation}\label{DougRamCKOP}\begin{split}
{}_5F_4&\left(\begin{array}{rrcccl}m-1,&\frac{m+1}{2},&-x\quad;&-y\quad,&-z\quad&\\&&&&&;1\\&\frac{m-1}{2},&x+m\:,&y+m\:,&z+m\:&\end{array}\right)\\
&=\frac{\Gamma(x+m)\Gamma(y+m)\Gamma(z+m)\Gamma(x+y+z+m)}{\Gamma(m)\Gamma(x+y+m)\Gamma(y+z+m)\Gamma(x+z+m)}.\end{split}\end{equation}

Letting $m=2$, $x=-1-\frac{\mu_1}{2}$, $y=-1-\frac{\mu_2}{2}$ and $z=-1-\frac{\mu_3}{2}$ in \eqref{DougRamCKOP} yields: \begin{equation}\label{tau1}\tag{$\dagger$}\tau_1=\left(2\pi^{\frac{3}{2}}\right)^3\Gamma\left(\frac{-\mu_1-\mu_2-\mu_3-2}{2}\right)\prod_{j=1}^3\frac{\Gamma\left(-\frac{1}{2}-\frac{\mu_j}{2}\right)}{\Gamma\left(\frac{\mu_j-\mu_1-\mu_2-\mu_3}{2}\right)}.\end{equation}

\subsection{\texorpdfstring{Case $n=2$}{Case n=2}}

In recent work, K. Fatawat and V. Yashoverdhan \cite{VyasFatawat} have extended Bailey's classical method to study generalized hypergeometric functions under certain compatibility conditions on the parameters. Of particular relevance here is the following generalization of Dougall's theorem (see Formula (4.12) in \cite{VyasFatawat} for the proof and \cite{Bailey} or \cite{Slater} for Dougall's original result):
 
\begin{equation}\label{VF4.12}\begin{split}
{}_6F_5&\left(\begin{array}{rrrcccl}f-1,&\frac{f}{2}+1,&\frac{f+1}{2},&u\quad,&v\quad,&w\quad&\\&&&&&&;1\\&\frac{f}{2}-1,&\frac{f-1}{2},&f-u\:,&f-v\:,&f-w\:&\end{array}\right)\\
&=\frac{\Gamma(f-u)\Gamma(f-v)\Gamma(f-w)\Gamma(f-u-v-w-1)}{\Gamma(f)\Gamma(f-v-w)\Gamma(f-u-v-1)\Gamma(f-u-w-1)}\cdot\kappa
\end{split}\end{equation}

where \[\kappa=\frac{vw-h(1+u+v+w-f)}{h(1+u+v-f)(1+u+w-f)}\qquad\text{and}\qquad h=\frac{f(f-2)}{4u}.\]

\begin{remark}Note that our $\kappa$ is the inverse of the quantity $k$ in \cite{VyasFatawat}.\end{remark}

Setting the parameters $f=2n=4$ and \[u=\frac{2n-\alpha}{4}\quad,\quad v=\frac{2n-\beta}{4}\quad,\quad w=\frac{2n-\gamma}{4}\] in \eqref{VF4.12}, so that $h=\dfrac{8}{4-\alpha}$ and \[\kappa=\frac{(4-\alpha)(4-\beta)(4-\gamma)+32(\alpha+\beta+\gamma)}{8(4+\alpha+\beta)(4+\alpha+\gamma)},\]
we obtain \[\tau_2=\frac{4\pi^{10}\sqrt{\pi}}{3}\frac{\Gamma\left(\frac{\alpha+\beta+\gamma}{4}\right)}{\Gamma\left(2+\frac{\beta+\gamma}{4}\right)\Gamma\left(1+\frac{\alpha+\beta}{4}\right)\Gamma\left(1+\frac{\alpha+\gamma}{4}\right)}\cdot \kappa\:,\] that is,

\[\tau_2=\frac{\pi^{10}\sqrt{\pi}}{6}\frac{\left((4-\alpha)(4-\beta)(4-\gamma)+32(\alpha+\beta+\gamma)\right)\Gamma\left(\frac{\alpha+\beta+\gamma}{4}\right)}{(4+\alpha+\beta)(4+\alpha+\gamma)\Gamma\left(2+\frac{\beta+\gamma}{4}\right)\Gamma\left(1+\frac{\alpha+\beta}{4}\right)\Gamma\left(1+\frac{\alpha+\gamma}{4}\right)},\]
which reduces to 
\begin{equation}\label{tau2}\tag{$\ddagger$}\tau_2=\frac{\pi^{10}\sqrt{\pi}}{96}\frac{\left((4-\alpha)(4-\beta)(4-\gamma)+32(\alpha+\beta+\gamma)\right)\Gamma\left(\frac{\alpha+\beta+\gamma}{4}\right)}{\Gamma\left(2+\frac{\beta+\gamma}{4}\right)\Gamma\left(2+\frac{\alpha+\beta}{4}\right)\Gamma\left(2+\frac{\alpha+\gamma}{4}\right)}.\end{equation}


\subsection{Concluding remarks}\label{SecDisc}

The construction of an invariant trilinear form obtained above by taking the trace of compositions of Knapp-Stein intertwiners follows the method used in the work presented in the introduction, in particular \cite{BernsteinRez}. As in the situations that were previously studied, the Bernstein-Reznikov integrals associated with spherical degenerate principal series of the complex symplectic group $\SpnC$ is can be expressed in terms of hypergeometric functions.

The existence of the closed formula \eqref{tau1} in the case $n=1$ is consistent with the results obtained in the conformal case studied in  \cite{BernsteinRez}, \cite{ClercOrsted} and \cite{DeitmarTriple}, in view of the isogeny between $\mathrm{Sp}(1,\C)=\mathrm{SL}(2,\C)$ and $\mathrm{SO}_\circ(3,1)$.

The proof of \cite{DeitmarTriple} for $\mathrm{SO}_\circ(d,1)$ relies on the use of a recursion formula together with the original result of \cite{BernstRezOriginal} for $d=2$ and the corresponding one for $d=3$, obtained along the same lines. The formula was established again in \cite{BernsteinRez} as the result of a comparison with invariant trilinear forms arising from the representation theory of real symplectic groups \cite[Proposition 4.1]{BernsteinRez}. It also falls under the general treatment of real-rank one groups exposed in \cite{BKZTrilinear}.

Our formula \eqref{tau2}, obtained in the case $n=2$ relies on recent results of Fatawat and Vyas \cite{VyasFatawat} on very well-poised hypergeometric functions. However, these results do not allow to treat higher order cases yet, and we were not able to find in the literature any useful reduction formula for the ${}_6F_5$ function that appears in the general expression of $\tau_n$ for $n\geq3$.

Note that there also exist low-dimensional isogenies between the complex symplectic groups $\mathrm{Sp}(1,\C)$ and $\mathrm{Sp}(2,\C)$ and the complex orthogonal groups $\mathrm{SO}(3,\C)$ and $\mathrm{SO}(5,\C)$ respectively (see \cite[p.352--355]{ECartan} or \cite[p. 519]{HelgasonDS}). This might indicate that interesting results could be expected for odd-dimensional complex orthogonal groups $\mathrm{SO}(2n+1,\C)$, although the methods used here do not allow to directly address this question in full generality as of now.

Finally, it would be interesting to extract the representation theoretic significance of the zeros and poles of the expression found in \eqref{tau2}.

\bigskip

\begin{center}
\textbf{Acknowledgements}
\end{center}

The author thanks Prof. Toshiyuki Kobayashi, who suggested to study the questions addressed in this article, and Prof. Michael Pevzner for many inspiring discussions.

\nocite{KobSpehCRAS,KobSpehMemAMS}
\newpage
\bibliographystyle{amsalpha}
\bibliography{biblio}
\end{document}